\documentclass[11pt,reqno]{amsart}
\usepackage{preprintStyle}
\usepackage{textcase}

\newcommand{\imagUnit}{i}

\newcommand{\N}{\ensuremath\mathbb{N}}

\newcommand{\R}{\ensuremath\mathbb{R}}
\newcommand{\C}{\ensuremath\mathbb{C}}

\newcommand{\Hinf}{\mathcal{H}_\infty}
\newcommand{\Htwo}{\mathcal{H}_2}

\newcommand{\Ltwo}{L_2}
\newcommand{\Lfour}{L_4}

\newcommand{\Lone}{L_1}
\newcommand{\Linfty}{L_\infty}

\newcommand{\T}{\ensuremath\mathsf{T}}

\newcommand*{\herm}{^{\mathsf{H}}}

\newcommand{\ds}{\,\mathrm{d}s}
\newcommand{\dhats}{\,\mathrm{d}\hat{s}}
\newcommand{\dt}{\,\mathrm{d}t}

\newcommand{\dtau}{\,\mathrm{d}\tau}
\newcommand{\dsigma}{\,\mathrm{d}\sigma}

\DeclareMathOperator{\real}{Re}

\DeclareMathOperator{\tr}{tr}
\DeclareMathOperator{\trace}{trace}

\newcommand{\mexp}[1]{\mathrm{e}^{#1}}

\newcommand{\calA}{\mathcal{A}}

\newcommand{\calL}{\mathcal{L}}

\newcommand{\calQ}{\mathcal{Q}}

\newcommand{\calY}{\mathcal{Y}}

\newcommand{\system}{\Sigma}
\newcommand{\state}{x}

\newcommand{\stateDim}{n}
\newcommand{\systemMat}[1]{\mathbf{#1}}

\newcommand{\kernel}{\mathbf{h}}
\newcommand{\transferFunction}{\mathbf{G}}

\newcommand{\inpVar}{u}
\newcommand{\inpVarFreq}{U}

\newcommand{\inpVarDim}{m}
\newcommand{\inputDim}{\inpVarDim}

\newcommand{\outVar}{y}
\newcommand{\outVarBiVar}{\widetilde{\outVar}}
\newcommand{\outVarFreq}{Y}
\newcommand{\outVarFreqBiVar}{\widetilde{\outVarFreq}}

\newcommand{\outputVar}{\outVar}
\newcommand{\outVarDim}{p}

\newcommand{\I}{\systemMat{I}}

\newcommand{\A}{\systemMat{A}}
\newcommand{\B}{\systemMat{B}}
\newcommand{\Cm}{\systemMat{C}}
\newcommand{\D}{\systemMat{D}}
\newcommand{\M}{\systemMat{M}}
\newcommand{\Mvec}{\boldsymbol{\mathcal{M}}}
\newcommand{\Pm}{\systemMat{P}}
\newcommand{\Pvec}{\boldsymbol{\mathcal{P}}}
\newcommand{\Rm}{\systemMat{R}}
\newcommand{\Rvec}{\boldsymbol{\mathcal{R}}}

\DeclareMathOperator{\vecc}{vec}

\DeclareMathOperator{\vtf}{vtf}

\definecolor{mycolor1}{rgb}{0.00000,0.44700,0.74100}
\definecolor{mycolor2}{rgb}{0.85000,0.32500,0.09800}
\definecolor{mycolor3}{rgb}{0.92900,0.69400,0.12500}
\definecolor{mycolor4}{rgb}{0.46600,0.67400,0.18800}
\definecolor{mycolor5}{rgb}{0.49400,0.18400,0.55600}

\definecolor{nicegreen}{rgb}{0.3,0.7,0.4}

\newcommand{\abbr}[1]{\textsf{#1}\xspace}

\newcommand{\LTI}{\abbr{LTI}}		
\newcommand{\LQO}{\abbr{LQO}}	

\newcommand{\BT}{\abbr{BT}} 		
\newcommand{\IRKA}{\abbr{IRKA}}		
\newcommand{\MOR}{\abbr{MOR}}		

\DeclareRobustCommand\trineg{%
\begin{tikzpicture}[baseline,scale=.07]
\draw (0,0) -- (4,0) -- (2,3.46410161514) -- cycle;
\draw [domain=1:3.4] plot(\x,{(5.6-3.46410161514*\x)/-2});
\end{tikzpicture}}

\newcommand{\vecNorm}[1]{\left\| #1\right\|}
\newcommand{\funNorm}[1]{\left|\!\left|\!\left| #1 \right|\!\right|\!\right|}
\newcommand{\freqSpacetwo   }[1]{\Ltwo  \left(\imagUnit\R,\, \R^{#1}  \right)}
\newcommand{\freqSpacetwoC  }[1]{\Ltwo  \left(\imagUnit\R,\, \C^{#1}  \right)}

\newcommand{\freqSpaceoneC   }[1]{\Lone  \left(\imagUnit\R,\, \C^{#1}  \right)}
\newcommand{\timeSpacetwo   }[1]{\Ltwo  \left([0, \infty), \, \R^{#1} \right)}
\newcommand{\timeSpacefour   }[1]{\Lfour  \left([0, \infty), \, \R^{#1} \right)}

\newcommand{\timeSpaceinf   }[1]{\Linfty\left([0, \infty), \, \R^{#1} \right)}

\newcommand{\intFreqSpace}{\int_{-\imagUnit \infty}^{\imagUnit \infty}}

\newcommand{\lyapM}{\calY}

\newcommand{\lyapBB}{\calQ}
\newcommand{\lyapBBAux}{\widehat{\lyapBB}}

\definecolor{kit-green}{RGB}{0, 150, 130}
\colorlet{kit-green100}{kit-green}
\colorlet{kit-green90}{kit-green!90!white}
\colorlet{kit-green80}{kit-green!80!white}
\colorlet{kit-green70}{kit-green!70!white}
\colorlet{kit-green60}{kit-green!60!white}
\colorlet{kit-green50}{kit-green!50!white}
\colorlet{kit-green40}{kit-green!40!white}
\colorlet{kit-green30}{kit-green!30!white}
\colorlet{kit-green25}{kit-green!25!white}
\colorlet{kit-green20}{kit-green!20!white}
\colorlet{kit-green15}{kit-green!15!white}
\colorlet{kit-green10}{kit-green!10!white}
\colorlet{kit-green5}{kit-green!5!white}

\definecolor{kit-blue}{RGB}{70, 100, 170}
\colorlet{kit-blue100}{kit-blue}
\colorlet{kit-blue90}{kit-blue!90!white}
\colorlet{kit-blue80}{kit-blue!80!white}
\colorlet{kit-blue70}{kit-blue!70!white}
\colorlet{kit-blue60}{kit-blue!60!white}
\colorlet{kit-blue50}{kit-blue!50!white}
\colorlet{kit-blue40}{kit-blue!40!white}
\colorlet{kit-blue30}{kit-blue!30!white}
\colorlet{kit-blue25}{kit-blue!25!white}
\colorlet{kit-blue20}{kit-blue!20!white}
\colorlet{kit-blue15}{kit-blue!15!white}
\colorlet{kit-blue10}{kit-blue!10!white}
\colorlet{kit-blue5}{kit-blue!5!white}

\definecolor{kit-royalblue}{RGB}{0, 45, 76}
\colorlet{kit-royalblue100}{kit-royalblue}
\colorlet{kit-royalblue90}{kit-royalblue!90!white}
\colorlet{kit-royalblue80}{kit-royalblue!80!white}
\colorlet{kit-royalblue70}{kit-royalblue!70!white}
\colorlet{kit-royalblue60}{kit-royalblue!60!white}
\colorlet{kit-royalblue50}{kit-royalblue!50!white}
\colorlet{kit-royalblue40}{kit-royalblue!40!white}
\colorlet{kit-royalblue30}{kit-royalblue!30!white}
\colorlet{kit-royalblue25}{kit-royalblue!25!white}
\colorlet{kit-royalblue20}{kit-royalblue!20!white}
\colorlet{kit-royalblue15}{kit-royalblue!15!white}
\colorlet{kit-royalblue10}{kit-royalblue!10!white}
\colorlet{kit-royalblue5}{kit-royalblue!5!white}

\definecolor{kit-iceblue100}{RGB}{30, 53, 69}
\definecolor{kit-iceblue70}{RGB}{68, 94, 111}
\definecolor{kit-iceblue50}{RGB}{168, 185, 196}
\definecolor{kit-iceblue30}{RGB}{218, 225, 230}

\definecolor{kit-red}{RGB}{162, 34, 35}
\colorlet{kit-red100}{kit-red}
\colorlet{kit-red90}{kit-red!90!white}
\colorlet{kit-red80}{kit-red!80!white}
\colorlet{kit-red70}{kit-red!70!white}
\colorlet{kit-red60}{kit-red!60!white}
\colorlet{kit-red50}{kit-red!50!white}
\colorlet{kit-red40}{kit-red!40!white}
\colorlet{kit-red30}{kit-red!30!white}
\colorlet{kit-red25}{kit-red!25!white}
\colorlet{kit-red20}{kit-red!20!white}
\colorlet{kit-red15}{kit-red!15!white}
\colorlet{kit-red10}{kit-red!10!white}
\colorlet{kit-red5}{kit-red!5!white}

\definecolor{kit-yellow}{RGB}{252, 229, 0}
\colorlet{kit-yellow100}{kit-yellow}
\colorlet{kit-yellow90}{kit-yellow!90!white}
\colorlet{kit-yellow80}{kit-yellow!80!white}
\colorlet{kit-yellow70}{kit-yellow!70!white}
\colorlet{kit-yellow60}{kit-yellow!60!white}
\colorlet{kit-yellow50}{kit-yellow!50!white}
\colorlet{kit-yellow40}{kit-yellow!40!white}
\colorlet{kit-yellow30}{kit-yellow!30!white}
\colorlet{kit-yellow25}{kit-yellow!25!white}
\colorlet{kit-yellow20}{kit-yellow!20!white}
\colorlet{kit-yellow15}{kit-yellow!15!white}
\colorlet{kit-yellow10}{kit-yellow!10!white}
\colorlet{kit-yellow5}{kit-yellow!5!white}

\definecolor{kit-orange}{RGB}{223, 155, 27}
\colorlet{kit-orange100}{kit-orange}
\colorlet{kit-orange90}{kit-orange!90!white}
\colorlet{kit-orange80}{kit-orange!80!white}
\colorlet{kit-orange70}{kit-orange!70!white}
\colorlet{kit-orange60}{kit-orange!60!white}
\colorlet{kit-orange50}{kit-orange!50!white}
\colorlet{kit-orange40}{kit-orange!40!white}
\colorlet{kit-orange30}{kit-orange!30!white}
\colorlet{kit-orange25}{kit-orange!25!white}
\colorlet{kit-orange20}{kit-orange!20!white}
\colorlet{kit-orange15}{kit-orange!15!white}
\colorlet{kit-orange10}{kit-orange!10!white}
\colorlet{kit-orange5}{kit-orange!5!white}

\definecolor{kit-lightgreen}{RGB}{140, 182, 60}
\colorlet{kit-lightgreen100}{kit-lightgreen}
\colorlet{kit-lightgreen90}{kit-lightgreen!90!white}
\colorlet{kit-lightgreen80}{kit-lightgreen!80!white}
\colorlet{kit-lightgreen70}{kit-lightgreen!70!white}
\colorlet{kit-lightgreen60}{kit-lightgreen!60!white}
\colorlet{kit-lightgreen50}{kit-lightgreen!50!white}
\colorlet{kit-lightgreen40}{kit-lightgreen!40!white}
\colorlet{kit-lightgreen30}{kit-lightgreen!30!white}
\colorlet{kit-lightgreen25}{kit-lightgreen!25!white}
\colorlet{kit-lightgreen20}{kit-lightgreen!20!white}
\colorlet{kit-lightgreen15}{kit-lightgreen!15!white}
\colorlet{kit-lightgreen10}{kit-lightgreen!10!white}
\colorlet{kit-lightgreen5}{kit-lightgreen!5!white}

\definecolor{kit-purple}{RGB}{163, 16, 124}
\colorlet{kit-purple100}{kit-purple}
\colorlet{kit-purple90}{kit-purple!90!white}
\colorlet{kit-purple80}{kit-purple!80!white}
\colorlet{kit-purple70}{kit-purple!70!white}
\colorlet{kit-purple60}{kit-purple!60!white}
\colorlet{kit-purple50}{kit-purple!50!white}
\colorlet{kit-purple40}{kit-purple!40!white}
\colorlet{kit-purple30}{kit-purple!30!white}
\colorlet{kit-purple25}{kit-purple!25!white}
\colorlet{kit-purple20}{kit-purple!20!white}
\colorlet{kit-purple15}{kit-purple!15!white}
\colorlet{kit-purple10}{kit-purple!10!white}
\colorlet{kit-purple5}{kit-purple!5!white}

\definecolor{kit-brown}{RGB}{167, 130, 46}
\colorlet{kit-brown100}{kit-brown}
\colorlet{kit-brown90}{kit-brown!90!white}
\colorlet{kit-brown80}{kit-brown!80!white}
\colorlet{kit-brown70}{kit-brown!70!white}
\colorlet{kit-brown60}{kit-brown!60!white}
\colorlet{kit-brown50}{kit-brown!50!white}
\colorlet{kit-brown40}{kit-brown!40!white}
\colorlet{kit-brown30}{kit-brown!30!white}
\colorlet{kit-brown25}{kit-brown!25!white}
\colorlet{kit-brown20}{kit-brown!20!white}
\colorlet{kit-brown15}{kit-brown!15!white}
\colorlet{kit-brown10}{kit-brown!10!white}
\colorlet{kit-brown5}{kit-brown!5!white}

\definecolor{kit-cyan}{RGB}{35, 161, 224}
\colorlet{kit-cyan100}{kit-cyan}
\colorlet{kit-cyan90}{kit-cyan!90!white}
\colorlet{kit-cyan80}{kit-cyan!80!white}
\colorlet{kit-cyan70}{kit-cyan!70!white}
\colorlet{kit-cyan60}{kit-cyan!60!white}
\colorlet{kit-cyan50}{kit-cyan!50!white}
\colorlet{kit-cyan40}{kit-cyan!40!white}
\colorlet{kit-cyan30}{kit-cyan!30!white}
\colorlet{kit-cyan25}{kit-cyan!25!white}
\colorlet{kit-cyan20}{kit-cyan!20!white}
\colorlet{kit-cyan15}{kit-cyan!15!white}
\colorlet{kit-cyan10}{kit-cyan!10!white}
\colorlet{kit-cyan5}{kit-cyan!5!white}

\definecolor{kit-gray}{RGB}{0, 0, 0}
\colorlet{kit-gray100}{kit-gray}
\colorlet{kit-gray90}{kit-gray!90!white}
\colorlet{kit-gray80}{kit-gray!80!white}
\colorlet{kit-gray70}{kit-gray!70!white}
\colorlet{kit-gray60}{kit-gray!60!white}
\colorlet{kit-gray50}{kit-gray!50!white}
\colorlet{kit-gray40}{kit-gray!40!white}
\colorlet{kit-gray30}{kit-gray!30!white}
\colorlet{kit-gray25}{kit-gray!25!white}
\colorlet{kit-gray20}{kit-gray!20!white}
\colorlet{kit-gray15}{kit-gray!15!white}
\colorlet{kit-gray10}{kit-gray!10!white}
\colorlet{kit-gray5}{kit-gray!5!white}

\newcommand{\removeAuthorData}[1]{#1}

\title[$\Ltwo$-$\Ltwo$-gain bounds for quadratic output systems]{$\Ltwo$-$\Ltwo$-gain bounds for quadratic output systems}

\removeAuthorData{
\author[Birgit Hillebrecht]{Birgit Hillebrecht \\ \protect\NoCaseChange{ORCID:\,0000-0001-5361-0505, Institute for Applied and Numerical Mathematics, Karlsruhe Institute of Technology, 76131 Karlsruhe, Germany, Email address: \texttt{birgit.hillebrecht@kit.edu}}}
}

\begin{document}

\begin{abstract}
  We derive an explicit bound for the $\Ltwo$-$\Ltwo$ gain of linear time-invariant systems whose output is a quadratic function of the state and the input. Such systems appear naturally in many areas, for example for port-Hamiltonian systems, optimal-control, and stochastic problems. In case the output is purely quadratic in the state, the bound is determined by the supremum over $\omega\in\R$ of the $\Ltwo$-norms of the bivariate transfer function restricted to the anti-diagonals $\{(s,\,\imagUnit \omega-s)\mid s\in\imagUnit\R\}$ of the frequency domain $\imagUnit\R\times\imagUnit\R$. Further, we show how the bound can be computed by solving linear matrix equations. This result provides a practical tool for assessing and reducing quadratic-output models. 
\end{abstract}

\maketitle

\smallskip



\section{Introduction}

We investigate linear time-invariant systems with quadratic outputs (\LQO systems) of the form 
\begin{equation}\label{eq:base_ext}
  \system \; \left\lbrace \;\begin{aligned}
    \dot{\state}(t) &= \A \state(t) + \B \inpVar(t), \\
    \state(0) &=0,\\
    \outputVar(t) &= \Cm \state(t) + \D \inpVar(t)+ \Mvec (\state(t) \otimes \state(t))  + \Rvec (\state(t) \otimes \inpVar(t)) + \Pvec (\inpVar(t) \otimes \inpVar(t))  ,
  \end{aligned} \right.
\end{equation} 
with $\A \in \R^{\stateDim \times \stateDim}$ Hurwitz, $\B \in \R^{\stateDim \times \inputDim}, \Cm \in \R^{p \times n}, \D \in \R^{p \times m}, \Mvec \in \R^{\outVarDim \times \stateDim^2}, \Rvec \in \R^{p \times mn},$ and $\Pvec \in \R^{p\times m^2}$. For $t\ge 0$, the state $\state(t)\in \R^n$ and the output $\outVar(t)\in \R^p$ evolve depending on the input $\inpVar(t) \in \R^m$. We assume that the input $\inpVar$ lies in $\Ltwo(\R_+)$ and that $\state(0)=0$ without further mention. The quadratic and bilinear terms in the output $\outVar(t)$ in \eqref{eq:base_ext} extend the standard linear time-invariant (\LTI) systems to a mildly nonlinear form. Such quadratic outputs arise naturally in several application areas: In port-Hamiltonian systems they reflect the input-energy relation \cite{HolNSU25}, in linear-quadratic optimal control they are part of the optimization objective \cite{EngW08}, and in stochastic problems the variance is a quadratic function of the states \cite{HaaUW13}. 

Motivated by these applications and the relevance of $\Ltwo$-$\Ltwo$ gains for robust and optimal control \cite{Sch92,ZhoDG96} as well as for model order reduction (\MOR) \cite{BenSGQR21}, the central aim of this manuscript is to derive a bound for the $\Ltwo$-$\Ltwo$ gain of the system \eqref{eq:base_ext}. It is well-known that, for stable \LTI systems, the $\Ltwo$-$\Ltwo$ gain is bounded by the $\Hinf$-norm of the transfer function \cite{BenSGQR21}. Nevertheless, an analogous expression for \LQO systems is not yet known. In particular, it is unclear how to compute $c_{2} \ge 0$  such that
\begin{equation*}
  \funNorm{ \outVar }_{\timeSpacetwo{p}} \le c_{1} \funNorm{ \inpVar }_{\timeSpacetwo{m}} + c_{2} \funNorm{ \inpVar }_{\timeSpacetwo{m}}^2,
\end{equation*} 
where $c_1$ is the $\Hinf$-norm of the transfer function of the linear part of the output.

Previous studies derived computable $\Ltwo$-$\Linfty$-gain bounds 
\begin{equation*}
  \funNorm{ \outVar }_{\timeSpaceinf{p}} \le c_{1} \funNorm{ \inpVar }_{\timeSpacetwo{m}} + c_{2} \funNorm{ \inpVar }_{\timeSpacetwo{m}}^2
\end{equation*}
and associated \MOR methods. Specifically, a Gramian-based formulation and a balanced truncation (\BT) algorithm were presented in \cite{BenGD21}, and an adaptation of the iterative rational Krylov algorithm (\IRKA) in \cite{ReiGPG26}. A key observation in these works is that, rather than the supremum of the original output, one may consider the supremum of a bivariate output
\begin{equation} \label{eq::boundsups}
  \sup_{t \ge 0} \vecNorm{ \outVar(t) }_2 \le \sup_{t_1, t_2 \ge 0} \vecNorm{\outVarBiVar(t_1, t_2) }_2 ,
\end{equation}
where $\outVarBiVar(t, t) = \outVar(t)$ for every $t \ge0 $. The two outputs $\outVar$ and $ \outVarBiVar$ are generated by the same kernel 
\begin{equation}
\kernel(\tau_1, \tau_2) = \Mvec \left[\mexp{\A \tau_1}\B \otimes \mexp{\A \tau_2}\B \right]+ \Rvec \left[\mexp{\A \tau_1}\B \otimes \I \delta(\tau_2)\right] + \Pvec\delta(\tau_1)\delta(\tau_2)\label{eq:kernel}
\end{equation}
 but by different integral representations:
\begin{subequations}
  \begin{align}
    \outVar(t) &= \int_0^t \int_0^t \kernel(\tau_1, \tau_2) \left(\inpVar(t-\tau_1) \otimes \inpVar(t-\tau_2)\right) \dtau_1 \dtau_2, \quad \text{and} \label{eq:outunivar}\\
    \outVarBiVar(t_1, t_2) &= \int_0^{t_1} \int_0^{t_2} \kernel(\tau_1, \tau_2) \left(\inpVar(t_1-\tau_1) \otimes \inpVar(t_2-\tau_2)\right) \dtau_1 \dtau_2. \label{eq:outbivar}
  \end{align}
\end{subequations}
The idea in \cite{DiaHGA23,ReiGPG26} is to then apply the bivariate Laplace transform $\calL_2$ \cite{Deb15} to the kernel yielding the bivariate transfer function 
\begin{equation} \label{eq:tf_bivar}
  \widetilde \transferFunction(s_1, s_2) = \Mvec \left[(s_1 \I - \A)^{-1} \B \otimes (s_2\I - \A)^{-1} \B \right] + \Rvec\left[(s_1\I - \A)^{-1}\B \otimes \I \right] + \Pvec \in \C^{p \times m^2}.
\end{equation}
We denote the reconstruction of a full $k\times l$ matrix for $k,l \in \N$ from a vector in $\R^{kl}$ or $\C^{kl}$ by $\vtf_{k\times l}$ and introduce $\M_i \coloneqq \vtf_{n\times n}(\mathbf{e}_i \Mvec)$, $\Rm_i \coloneqq \vtf_{m \times n}(\mathbf{e}_i \Rvec)$, and $\Pm_i \coloneqq \vtf_{m \times m}(\mathbf{e}_i \Pvec)$ for all $i \in \{1,...,p\}$. Using this, we can write each of the $p$ rows of the bivariate transfer function as
\begin{align} 
  \widetilde \transferFunction_i(s_1, s_2) &\coloneqq \mathbf{e}_i \widetilde \transferFunction(s_1, s_2)\notag \\ &= \vecc\left(\B^\T (s_2\I - \A^\T)^{-1}\M_i (s_1 \I - \A)^{-1} \B + \Rm_i (s_1 \I - \A)^{-1} \B + \Pm_i\right)^\T.\label{eq:fullvectf}
\end{align}
Without loss of generality, the matrices $\M_i$ and $\Pm_i$ are assumed to be symmetric for the remainder of the manuscript. 

While the inequality \eqref{eq::boundsups} holds for the supremum of the outputs, an analogous bound for the $\Ltwo$-norm does not exist. To circumvent this, we apply the associated transform \cite{CheC73} in order to derive a suitable univariate transfer function and output in the frequency domain from their bivariate correspondences. 
The associated transform has been used for systems whose state equation contains bilinear and quadratic terms in \cite{ZhaW15}, but has not yet been leveraged for quadratic output systems.

The main contributions of this paper are grouped into two stages: In the first stage, we consider systems with $\Mvec\neq 0$ but $\Cm=0, \D = 0, \Rvec=0$ and $\Pvec=0$ and show in \Cref{thm:central} that the $\Ltwo$-$\Ltwo$-gain bound is determined by the supremum over all shifts $\omega\in\R$ of the $\Ltwo$-norm of the bivariate transfer function restricted to the anti-diagonal $\{(s,\,\imagUnit\omega-s)\mid s\in\imagUnit\R\}$. This  formulation leads to a corresponding semi-inner product definition (see \Cref{def:innerprod}) and thereby to an expression for this bound for the error between two \LQO systems, which is relevant in \MOR. We then show in \Cref{thm:lyap_formulation} that the semi-inner product can be evaluated by solving a set of Sylvester equations. In the second stage, we investigate the bilinear and quadratic throughput terms in \Cref{thm:onlyR} and \Cref{sec:throughput}. Finally, we summarize the results for general systems including bilinear terms as well as the known results for \LTI systems in \Cref{thm:full}. The paper concludes with a brief numerical example in \Cref{sec:example} and an outlook in \Cref{sec:discussion}.

\textbf{Notation.} From now on we follow the convention that bivariate auxiliary functions carry a
tilde $\tilde{\cdot}$, that frequency-domain quantities are written with uppercase letters, and that time-domain quantities are written with lowercase letters.
Further, we denote vector and matrix norms by $\vecNorm{\cdot}$, exemplarily, the Frobenius norm is referred to as $\vecNorm{\cdot}_F$ and the 2-norm as $\vecNorm{\cdot}_2$, while referring to function norms by $\funNorm{\cdot}$. We denote the reconstruction of a full $k\times l$ matrix for $k,l \in \N$ from a vector in $\R^{kl}$ or $\C^{kl}$ by $\vtf_{k\times l}$, where we drop the index if it is clear from context. The inverse operation, the vectorization, is denoted by $\vecc(\cdot)$. The $i$-th unit vector in $\R^l$ for $l\in \{p, m\}$ is denoted by $\mathbf{e}_i$.

\section{$\Ltwo$-$\Ltwo$-gain bound for the quadratic dependency on the state}

In this section, we consider the subclass of systems \eqref{eq:base_ext} whose output depends solely quadratically on the state, i.e., 
\begin{equation}\label{eq:base}
  \system_{\Mvec} \; \left\lbrace \;\begin{aligned}
    \dot{\state}(t) &= \A \state(t) + \B \inpVar(t), \\
    \outputVar(t) &= \Mvec (\state(t) \otimes \state(t)). 
  \end{aligned} \right.
\end{equation}
For this restricted case, the bivariate transfer function and output are given by \cite{DiaHGA23}
\begin{subequations}
  \begin{align}
    \label{eq:tf_only_M}
    \widetilde{\transferFunction}(s_1, s_2) &= \Mvec \left[(s_1 \I - \A)^{-1} \B \otimes (s_2\I - \A)^{-1} \B \right] \in \C^{p\times m^2},\\
    \outVarFreqBiVar(s_1, s_2) &= \widetilde{\transferFunction}(s_1, s_2) \left[\inpVarFreq(s_1) \otimes \inpVarFreq(s_2)\right] \in \C^p. \label{eq:out_bivar}
  \end{align}
\end{subequations}

The following \Cref{thm:central} establishes $\Ltwo$-$\Ltwo$-gain for the system \eqref{eq:base}. Its proof relies on two auxiliary results: First, \Cref{lem:assoc} shows how the univariate transfer function and output in the frequency domain can be computed on the imaginary axis from the bivariate correspondences. Second, \Cref{lem:y-twonorm-bound-Y-onenorm} rewrites the $\Ltwo$-norm of the output $\outVar$ as the
$\Lone$-norm of the bivariate frequency-domain output $\outVarFreqBiVar$. 

\begin{theorem}\label{thm:central}
 Consider the \LQO  system $\system_{\Mvec}$ in \eqref{eq:base}, then, the $\Ltwo$-norm of the output $\outVar$ is bounded by the $\Ltwo$-norm of the input $\inpVar$ by
  \begin{equation*}
    \funNorm{\outVar}_{\timeSpacetwo{p}}^2 \le \frac{1}{2 \pi}\sum_{i=1}^p \sup_{s\in \imagUnit \R} \funNorm{\widetilde \transferFunction_{i}(\cdot, s-\cdot)}_{\freqSpacetwoC{\inpVarDim^2}}^2 \funNorm{\inpVar}_{\timeSpacetwo{m}}^4,
  \end{equation*}
  where
  \begin{equation}
    \widetilde{\transferFunction}_i(s_1, s_2) = \vecc\left(\B^\T (s_2\I - \A^\T)^{-1} \M_i (s_1 \I - \A)^{-1} \B\right)
  \end{equation}
  where $\M_i = \vtf(\mathbf{e}_i \Mvec)$.
\end{theorem}

\begin{lemma}\label{lem:assoc}
  The univariate transfer function $\transferFunction(s) = \calL[\kernel](s)$ and the output in the frequency space $\outVarFreq(s) = \calL[\outVar](s)$ for \eqref{eq:base} can be computed for $s \in \imagUnit \R$ as 
  \begin{subequations}
    \begin{align}\label{eq:assoc_trafo}
      \transferFunction(s) & = \frac{1}{2\pi\imagUnit} \int_{-\imagUnit\infty}^{\imagUnit\infty} \widetilde{\transferFunction}(\sigma, s-\sigma) \dsigma, \quad \text{and}\\
      \outVarFreq(s) &= \frac{1}{2\pi\imagUnit} \int_{-\imagUnit\infty}^{\imagUnit\infty} \outVarFreqBiVar(\sigma, s-\sigma) \dsigma. \label{eq:assoc_trafo_y}
    \end{align}
  \end{subequations}
  In particular, $\outVar \in \Lone(\R_+, \R^p)$.
\end{lemma}
\begin{proof}
  For simplicity of presentation, we consider first 
  \begin{align*}
    \widetilde \Gamma(s_1, s_2) \coloneqq (s_1\I-\A)^{-1} \B \otimes (s_2\I-\A)^{-1} \B,
  \end{align*}
  for which each component can easily be identified to be compositional, i.e., of the form $f(s_1, s_2) = f_1(s_1)\cdot f_2(s_2)$. For compositional functions, the associated transform is given in \cite[Sec. 4]{CheC73} as
  \begin{equation*}
    \calA_2[f](s) = \frac{1}{2\pi \imagUnit}\int_{c-\imagUnit \infty}^{c+\imagUnit\infty} f_1(\sigma) f_2(s-\sigma) \dsigma
  \end{equation*}
  where $c \in \R$ such that the integration path lies entirely in the region where $f_1$ and $f_2(s-\cdot)$ are analytic. 
  
  Since $\A$ is assumed to be Hurwitz, we follow that $\widetilde{\Gamma}_1(\sigma) \coloneqq (\sigma \I-\A)^{-1} \B$ is analytic in the closed right complex half-plane, whereas $\widetilde{\Gamma}_2(s-\sigma) \coloneqq ((s-\sigma) \I-\A)^{-1} \B$ is analytic for $\real(\sigma)\le \real(s)$. Hence, for $\real(s)=0$ the imaginary axis is guaranteed to lie within the region of convergence of the integrand. As a consequence, the associated transform is well-defined for $\widetilde{\Gamma}$ and due to linearity also for $\widetilde{\transferFunction}(s_1, s_2) = \Mvec \widetilde{\Gamma}(s_1, s_2)$. Then, the result \cite[Sec. 4]{CheC73} yields the equality $\transferFunction(s) = \calA_2[\widetilde \transferFunction](s)$ for $s \in \imagUnit \R$.
\end{proof}

\begin{lemma}\label{lem:y-twonorm-bound-Y-onenorm}
  For an \LQO  system $\system_{\Mvec}$ in \eqref{eq:base} the $\Ltwo$-norm of the output $\outVar$ is bounded by the $\Lone$-norms of the bivariate output in the frequency space $\outVarFreqBiVar(s_1, s_2) = \calL_2(\outVarBiVar)(s_1, s_2)$ along the anti-diagonals as follows 
  \begin{align*}
    \funNorm{\outVar}_{\timeSpacetwo{p}}^2 &\le 
    \frac{1}{(2\pi)^3} \intFreqSpace \sum_{i=1}^p \funNorm{\outVarFreqBiVar_i(\cdot, s-\cdot)}_{\freqSpaceoneC{}}^2 \ds
  \end{align*}
\end{lemma}

\begin{proof}
  We compute the $\Ltwo$-norm of the output in the time domain $\outVar$ by applying Parseval's theorem in $(P)$ and computing the $\Ltwo$-norm of the output in the frequency domain $\outVarFreq$. Denoting the $i$-th unit vector in $\R^p$ by $\mathbf{e}_i$ this yields
  \begin{equation*}
    \begin{aligned}
      \funNorm{\outputVar}_{\timeSpacetwo{\outVarDim}}^2 &\overset{(P)}{=} \frac{1}{2\pi}\funNorm{\outVarFreq}_{\freqSpacetwo{\outVarDim}}^2 \\
      &\!\!\!\!\!\!\!\!\overset{\Cref{lem:assoc}}{=}\frac{1}{2\pi} \funNorm{\frac{1}{2\pi \imagUnit}\intFreqSpace \outVarFreqBiVar(\sigma, \cdot-\sigma) \dsigma}_{\freqSpacetwo{\outVarDim}}^2\\
      &\;\le \frac{1}{2\pi}\funNorm{\frac{1}{2\pi \imagUnit}\intFreqSpace \sum_{i=1}^p \left| \outVarFreqBiVar_i(\sigma, \cdot-\sigma)\right| \mathbf{e}_i \dsigma}_{\freqSpacetwo{\outVarDim}}^2\\
      &\;= \frac{1}{(2\pi)^3} \intFreqSpace \sum_{i=1}^p \funNorm{\outVarFreqBiVar_i(\cdot, s-\cdot)}_{\freqSpaceoneC{}}^2 \ds. \qquad\qquad\qquad\qedhere
    \end{aligned}
  \end{equation*}
\end{proof}

Using these two preliminary results, we can continue to proof \Cref{thm:central}.

\begin{proof}[Proof of \Cref{thm:central}]
  Starting from the result of \Cref{lem:y-twonorm-bound-Y-onenorm}, we can apply to the $\Lone$-norm of the fiber of $\outVarFreqBiVar_i(\sigma, s-\sigma)$ for fixed $s \in \imagUnit \R$ the Cauchy--Schwarz inequality on \eqref{eq:out_bivar} 
  \begin{equation*}
    \funNorm{\outVarFreqBiVar_i(\cdot, s-\cdot)}_{\Lone(\imagUnit\R, \C)} \le \funNorm{\widetilde\transferFunction_{i}(\cdot, s-\cdot)}_{\freqSpacetwoC{\inpVarDim^2}} \funNorm{\inpVarFreq(\cdot) \otimes \inpVarFreq(s-\cdot)}_{\freqSpacetwoC{\inpVarDim^2}},
  \end{equation*}
  to yield for the full output
  \begin{equation*}
    \funNorm{\outVar}^2_{\timeSpacetwo{\outVarDim}}  \le \frac{1}{(2\pi)^3}\intFreqSpace \sum_{i=1}^p \funNorm{\widetilde\transferFunction_{ i}(\cdot, s-\cdot)}_{\freqSpacetwoC{\inpVarDim^2}}^2 \funNorm{\inpVarFreq(\cdot) \otimes \inpVarFreq(s-\cdot)}_{\freqSpacetwoC{\inpVarDim^2}}^2 \ds.
  \end{equation*}
  We subsequently apply another Hölder inequality, this time to the $\Lone$-norm of the product of the $\Ltwo$-fiber norms. The precise inequality reads
  \begin{align*}
    & \intFreqSpace\funNorm{\widetilde \transferFunction_{ i}(\cdot, s-\cdot)}^2_{\Ltwo(\imagUnit \R, \C^{m^2})} \funNorm{\inpVarFreq(\cdot ) \otimes \inpVarFreq(s-\cdot)}^2_{\Ltwo(\imagUnit \R, \C^{m^2})}\ds \\
    &\le \sup_{s \in \imagUnit \R}\left( \funNorm{\widetilde \transferFunction_{i}(\cdot, s-\cdot)}^2_{\Ltwo(\imagUnit \R, \C^{m^2})} \right) \intFreqSpace \funNorm{\inpVarFreq(\cdot ) \otimes \inpVarFreq(s-\cdot)}^2_{\Ltwo(\imagUnit \R, \C^{m^2})} \ds,
  \end{align*}
  for which we can compute the second factor as 
  \begin{align*}
    \intFreqSpace \funNorm{\inpVarFreq(\cdot ) \otimes \inpVarFreq(s-\cdot)}^2_{\Ltwo(\imagUnit \R, \C^{m^2})} \ds &=\intFreqSpace \intFreqSpace \vecNorm{\inpVarFreq(\sigma ) \otimes \inpVarFreq(s-\sigma)}_2^2 \dsigma \ds\\
    &\overset{(\ast)}{=}\intFreqSpace \intFreqSpace \vecNorm{\inpVarFreq(\sigma )}_2^2 \vecNorm{\inpVarFreq(s-\sigma)}_2^2 \dsigma \ds\\
    &\!\!\!\!\!\!\!\overset{(\hat{s} \coloneqq s - \sigma)}{=}\intFreqSpace \intFreqSpace \vecNorm{\inpVarFreq(\sigma )}_2^2 \vecNorm{\inpVarFreq(\hat{s})}_2^2 \, \mathrm{d}\hat{s}\dsigma \\
    &= \funNorm{\inpVarFreq}_{\Ltwo(\imagUnit\R, \C^m)}^4
  \end{align*}
  where in $(\ast)$ we used that $\vecNorm{\mathbf{a} \otimes \mathbf{b}}_2 = \vecNorm{\mathbf{a}}_2 \vecNorm{\mathbf{b}}_2$ and separated the two integrals after the substitution.
  Finally, this lets us further estimate the integral by
  \begin{align*}
  \funNorm{\outVar}^2_{\timeSpacetwo{\outVarDim}} &\le \frac{1}{(2\pi)^3} \sum_{i=1}^p\left(\sup_{s\in \imagUnit \R} \funNorm{\widetilde \transferFunction_{i}(\cdot, s-\cdot)}_{\freqSpacetwoC{\inpVarDim^2}}\right)^2 \funNorm{\inpVarFreq}_{\freqSpacetwoC{\inpVarDim}}^4 \\
  &= \frac{1}{2\pi}\sum_{i=1}^p\sup_{s\in \imagUnit \R}  \funNorm{\widetilde \transferFunction_{i}(\cdot, s-\cdot)}_{\freqSpacetwoC{\inpVarDim^2}}^2 \funNorm{\inpVar}_{\timeSpacetwo{\inputDim}}^4,
\end{align*}
where the last equality holds by Parseval's theorem.
\end{proof}

\begin{remark}
  The bound obtained in Theorem~\ref{thm:central} is in general \emph{not} the $\Hinf$-norm of the (associated) univariate transfer function \eqref{eq:assoc_trafo}
  \begin{equation*}
    \funNorm{\transferFunction}_{\Hinf} = \sup_{s \in \imagUnit\R} \vecNorm{\frac{1}{2\pi\imagUnit} \int_{-\imagUnit\infty}^{\imagUnit\infty}\widetilde{\transferFunction}(\sigma, s-\sigma) \dsigma}_2. 
  \end{equation*}
\end{remark}

\section{Semi-inner product and Sylvester-equation-based computation of the gain bound}

To obtain a computable expression for the output error between two \LQO-systems we first introduce a suitable semi-inner product in \Cref{def:innerprod}. \Cref{thm:lyap_formulation} then shows how this semi-inner product can be evaluated by solving a set of Sylvester equations. 

\begin{definition}\label{def:innerprod}
  For two \LQO systems \eqref{eq:base} given by $\Sigma_k = (\A_k, \B_k, \Mvec_k)$ with $\A_k$ Hurwitz for $k=1,2$ with bivariate transfer functions $\widetilde \transferFunction_1, \widetilde \transferFunction_2$ of the form \eqref{eq:tf_only_M}, we define for $\omega \in \R^p$ the following semi-inner product 
  \begin{subequations}
    \begin{align}
      \left\langle \widetilde \transferFunction_1, \widetilde \transferFunction_2 \right\rangle_{\omega}&\coloneqq \intFreqSpace \sum_{i=1}^p \trace\left(\systemMat{Z}_{1,i} (\omega, s) \systemMat{Z}\herm_{2,i}(\omega, s)\right)   \ds \label{eq:inner_prod_zz} , \quad \text{with}\\
      \systemMat{Z}_{k,i}(\omega, s) &\coloneqq  \left[ \B_k^\T \otimes \B_k^\T \right] \left[(s\I - \A_k^\T)^{-1} \otimes ((\imagUnit\omega_i-s)\I - \A_k^\T)^{-1}\right]  \Mvec_{k}^\T \mathbf{e}_i, \quad \text{for}\; k=1,2. \label{eq:inner_prod}
    \end{align}
  \end{subequations}
\end{definition}

\begin{remark}
  One can verify analogously to \cite[Def. 3.1]{BenGD21} that \eqref{eq:inner_prod} defines a semi-inner product with the induced seminorm
  \begin{equation}\label{eq:norm2}
    \funNorm{ \widetilde \transferFunction }^2_{\omega} \coloneqq \left\langle \widetilde \transferFunction, \widetilde \transferFunction\right\rangle_{\omega},
  \end{equation}
  and that the supremum over $\omega \in \R^p$ of \eqref{eq:norm2} is the $\Ltwo$-$\Ltwo$-bound from \Cref{thm:central} up to a factor $2\pi$.
\end{remark}

The semi-inner product can be evaluated by solving the matrix equations stated in the following theorem.

\begin{theorem}\label{thm:lyap_formulation}
  For two \LQO systems \eqref{eq:base} given by $\Sigma_k = (\A_k, \B_k, \Mvec_k)$ with state dimensions $n_k$ and bivariate transfer functions $\widetilde \transferFunction_k$ for $k=1,2$, the semi-inner product from \Cref{def:innerprod} for $\omega \in \R^p$ is given by
  \begin{align*}
   \left\langle \widetilde \transferFunction_1, \widetilde \transferFunction_2 \right\rangle_{\omega} &=  -4 \pi \sum_{i=1}^p \trace \left(\B^\T_1\lyapM_{1,i, \omega} \lyapBBAux \lyapM_{2,i, \omega}\herm \B_2 \right)
  \end{align*}
  where $\lyapBBAux \in \R^{n_1 \times n_2} \text{ and } \lyapM_{k,i, \omega} \in \C^{n_k \times n_k}$ for $k=1,2$ and $i=1,...,p$ solve the Sylvester equations
  \begin{subequations}
    \begin{align}
      (\A_k^\T- \imagUnit \omega_i\I) \lyapM_{k,i, \omega} + \lyapM_{k,i, \omega} \A_k &= \M_{k,i},\label{eq:lyapM}\\
      {\A_1} \lyapBBAux + \lyapBBAux \A_2^\T &= \B_1\B_2^\T,\label{eq:lyapBBAux}
    \end{align}
  \end{subequations}
  where $\vecc(\M_{k,i})^\T = \mathbf{e}_i^\T \Mvec_{k}$.
\end{theorem}

The proof requires a matrix-valued partial-fraction decomposition of the Kronecker product of the residuals presented in the subsequent lemma for which the proof can be found in the appendix \Cref{app:partial-fraction}.

\begin{proposition}\label{lem:partial-fraction-decomposition}
  For $\A \in \R^{\stateDim \times \stateDim}$ Hurwitz, $\omega \in \R$, and $s \in \C\setminus (\sigma(\A) \cup \sigma(-\A + \imagUnit \omega \I))$, where $\sigma(\A)$ denotes the set of eigenvalues of $\A$, the matrix $\left[\I \otimes (\A^\T-\imagUnit \omega \I) + \A^\T \otimes \I\right]$ is invertible and it holds that
  \begin{multline*}
    \left[(s\I - \A^\T)^{-1} \otimes ((\imagUnit \omega -s)\I - \A^\T)^{-1}\right]  \\ =   \left[-(s\I - \A^\T)^{-1} \otimes \I - \I \otimes ((\imagUnit \omega -s) \I - \A^\T)^{-1}\right] \left[\I \otimes (\A^\T-\imagUnit \omega \I) + \A^\T \otimes \I\right]^{-1}.
  \end{multline*}
\end{proposition}

\begin{proof}[Proof of \Cref{thm:lyap_formulation}]
  The partial fraction decomposition \Cref{lem:partial-fraction-decomposition} applied to $\systemMat{Z}_{k,i}(s)$ defined in \eqref{eq:inner_prod} yields
  \begin{align*}
    \systemMat{Z}_{k,i}(\omega, s) &= \left[ \B_k^\T \otimes \B_k^\T\right] \left[-(s\I-\A_k^\T)^{-1} \otimes \I - \I \otimes \left((\imagUnit \omega_i - s)\I  - \A_k^\T \right)^{-1}\right]\cdot \dots \\
    & \quad \dots \cdot \left[\I \otimes (\A_k^\T- \imagUnit \omega_i \I ) + \A_k^\T \otimes \I \right]^{-1} \vecc(\M_{k,i})\\
    &\overset{(\ast)}{=}\left[ \B_k^\T \otimes \B_k^\T\right] \left[-(s\I-\A_k^\T)^{-1} \otimes \I - \I \otimes \left((\imagUnit \omega_i - s)\I  - \A_k^\T \right)^{-1}\right]\vecc(\lyapM_{k,i,\omega})\\
    &= -\left[\B_k^\T(s\I-\A_k^\T)^{-1} \otimes \B_k^\T + \B_k^\T \otimes \B_k^\T\left((\imagUnit \omega_i - s)\I  - \A_k^\T \right)^{-1}\right]\vecc(\lyapM_{k,i,\omega})\\
    &= -\vecc\left(\B_k^\T\lyapM_{k,i,\omega} (s\I-\A_k)^{-1} \B_k\right) - \vecc\left(\B_k^\T \left((\imagUnit \omega_i - s)\I  - \A_k^\T \right)^{-1}\lyapM_{k,i,\omega}\B_k\right)\\
    &=\vcentcolon \systemMat{t}_{k,i}(\omega, s) + \systemMat{s}_{k,i}(\omega, s).
  \end{align*}
  Here, we inserted in $(*)$ the Kronecker-vectorization identity \cite[eq.~(273)]{PetP12} for the solution of the Sylvester equation \eqref{eq:lyapM}. We named the two summands $\systemMat{t}_{k,i}(\omega, s)$ and $\systemMat{s}_{k,i}(\omega, s)$ to facilitate further computation. We observe that since the shifts determined by $\omega$ are purely imaginary, $\systemMat{t}_{k,i}(\omega, s)$ is analytic in the right complex half plane, while $\systemMat{s}_{k,i}(\omega, s)$ is analytic in the left complex half plane. This means that for the integral over the product occurring in \eqref{eq:inner_prod_zz}, i.e.,
  \begin{align*}
    \left\langle \widetilde \transferFunction_1, \widetilde \transferFunction_2 \right\rangle_{\omega}&=  \intFreqSpace \sum_{i=1}^p \trace\left(\systemMat{Z}_{1,i} (\omega, s) \systemMat{Z}\herm_{2,i}(\omega, s)\right)   \ds \\
    &= \sum_{i=1}^p\tr\left(\intFreqSpace \systemMat{Z}_{1,i} (\omega, s) \systemMat{Z}\herm_{2,i}(\omega, s) \ds\right)\\
    &= \sum_{i=1}^p\tr\left(\intFreqSpace \left[\systemMat{t}_{1,i} (\omega, s) + \systemMat{s}_{1,i} (\omega, s)\right] \left[\systemMat{t}_{2,i} (\omega, s) + \systemMat{s}_{2,i} (\omega, s)\right]\herm \ds\right)\\
    &\overset{(*)}{=} \sum_{i=1}^p\tr\left(\intFreqSpace \systemMat{t}_{1,i} (\omega, s) \systemMat{t}_{2,i}\herm (\omega, s) + \systemMat{s}_{1,i} (\omega, s)\systemMat{s}_{2,i}\herm(\omega, s)\ds\right),
  \end{align*}
  the terms $\systemMat{t}_{1,i}\herm(\omega, s)\systemMat{s}_{2,i}(\omega, s)$ as well as $\systemMat{s}_{1,i}\herm(\omega, s)\systemMat{t}_{2,i}(\omega, s)$ vanish in $(\ast)$ leveraging suitable path integrals around the left and right complex half-plane and the residue theorem (see \cite{ReiGPG26} for details). Now, we can further evaluate the summands to
  \begin{align*}
    &\phantom{=} \intFreqSpace \tr\left(\systemMat{t}_{1, i}(\omega, s) \systemMat{t}_{2,i}\herm(\omega, s)\right)\ds\\ &=  \intFreqSpace \systemMat{t}_{2, i}\herm(\omega, s) \systemMat{t}_{1,i}(\omega, s) \ds\\
    &= \intFreqSpace \tr\left(\B_2^\T (-s\I - \A_2^\T)^{-1} \lyapM\herm_{2, i, \omega} \B_2 \B_1^\T \lyapM_{1, i, \omega} (s\I - \A_1)^{-1} \B_1 \right)\ds\\
    &= \intFreqSpace \tr\left(\B_1^\T \lyapM_{1, i, \omega} (s\I - \A_1)^{-1} \B_1 \B_2^\T (-s\I - \A_2^\T)^{-1} \lyapM\herm_{2, i, \omega} \B_2 \right)\ds\\
    &= \tr\left(\B_1^\T \lyapM_{1, i, \omega} \left[ \intFreqSpace  (s\I - \A_1)^{-1} \B_1 \B_2^\T (-s\I - \A_2^\T)^{-1} \ds\right] \lyapM\herm_{2, i, \omega} \B_2 \right),
  \end{align*}
  and
  \begin{align*}
    &\phantom{=} \intFreqSpace \tr\left(\systemMat{s}_{1, i}(\omega, s) \systemMat{s}_{2,i}\herm(\omega, s)\right)\ds\\ &=  \intFreqSpace \systemMat{s}_{2, i}\herm(\omega, s) \systemMat{s}_{1,i}(\omega, s) \ds\\
    &= \intFreqSpace \tr\left(\B_2^\T \lyapM\herm_{2, i, \omega} ((s-\imagUnit \omega)\I - \A_2)^{-1}  \B_2 \B_1^\T  ((\imagUnit \omega - s)\I - \A_1^\T)^{-1} \lyapM_{1, i, \omega} \B_1 \right)\ds\\
    &\!\!\!\!\!\!\!\overset{\hat{s}=s-\imagUnit \omega}{=} \intFreqSpace\tr\left(\B_2^\T \lyapM\herm_{2, i, \omega} (\hat{s}\I - \A_2)^{-1}  \B_2 \B_1^\T  (-\hat{s}\I - \A_1^\T)^{-1} \lyapM_{1, i, \omega} \B_1 \right)\dhats\\
    &= \intFreqSpace\tr\left(\B_1^\T \lyapM_{1, i, \omega}^\T (-\hat{s}\I - \A_1)^{-1}  \B_1 \B_2^\T  (\hat{s}\I - \A_2^\T)^{-1} \overline{\lyapM_{2, i, \omega}} \B_2 \right)\dhats\\
    &= \tr\left(\B_1^\T \lyapM^\T_{1, i, \omega} \left[ \intFreqSpace  (-\hat{s}\I - \A_1)^{-1} \B_1 \B_2^\T (\hat{s}\I - \A_2^\T)^{-1} \dhats\right] \overline{\lyapM_{2, i, \omega}} \B_2 \right).
  \end{align*}
  We can now identify the inner integral of both terms with the solution $\lyapBBAux$ of the Sylvester equation \eqref{eq:lyapBBAux}, i.e., a cross gramian between the two systems,
  \begin{equation*}
    \lyapBBAux = -\frac{1}{2\pi}\intFreqSpace  (-\hat{s}\I - \A_1)^{-1} \B_1 \B_2^\T (\hat{s}\I - \A_2^\T)^{-1} \dhats.
  \end{equation*}
  Further, we observe that \eqref{eq:lyapM} can be equivalently written as 
  \begin{equation*}
    (\A_k^\T - \frac{\imagUnit \omega_i}{2}\I ) \lyapM_{k,i, \omega} + \lyapM_{k, i, \omega}(\A_k - \frac{\imagUnit \omega_i}{2}\I ) = \M_{k,i}
  \end{equation*}
  and since $\M_{k,i}$ is symmetric, $\lyapM_{k, i, \omega}$ is also symmetric and the complex conjugate is equal to the hermitian transpose, i.e., $\overline{\lyapM_{k, i, \omega}} = \lyapM\herm_{k,i, \omega}$. Adding the terms yields the desired result.
\end{proof}

\begin{remark}
  Let two \LQO systems \eqref{eq:base} $\Sigma_1 = (\A_1, \B_1, \Mvec_1)$ and $\Sigma_2 = (\A_2, \B_2, \Mvec_2)$ be given, then the $\omega$-error can be computed as 
  \begin{equation*}
    \funNorm{ \widetilde\transferFunction_1- \widetilde\transferFunction_2 }^2_{\omega} = \funNorm{\widetilde\transferFunction_1}_{\omega}^2 + \funNorm{ \widetilde{\transferFunction}_2 }_{\omega}^2 - 2 \real \left\langle \widetilde \transferFunction_1, \widetilde \transferFunction_2 \right\rangle_{\omega}.
  \end{equation*}
\end{remark}

\begin{remark}
  Inspired by the $\Htwo$-bounding balanced truncation algorithm for \LQO systems in \cite{BenGD21}, the Sylvester-based formulation in \Cref{thm:lyap_formulation} can yield an $\Ltwo$-$\Ltwo$-oriented \BT in case the observability energy can be expressed through the matrices $\lyapBB$ and $\lyapM_i$. In this case, this is also complicated by the supremum which has to be taken over the $\omega$-norms and therefore no similar relation is obviously derivable.
\end{remark}

\section{Bilinear outputs}

In this section, we will consider systems which have only bilinear outputs, i.e., systems of the form 
\begin{equation}\label{eq:onlyR}
  \system_{\Rvec} \; \left\lbrace \;\begin{aligned}
    \dot{\state}(t) &= \A \state(t) + \B \inpVar(t), \\
    \outputVar(t) &=  \Rvec (\state(t) \otimes \inpVar(t)),
  \end{aligned} \right.
\end{equation} 
for $\A \in \R^{\stateDim\times \stateDim}$ Hurwitz, $\B \in \R^{\stateDim \times \inputDim}$, and $\Rvec \in \R^{\outVarDim \times \stateDim \inputDim}$.

\begin{theorem} \label{thm:onlyR}
  For a linear time-invariant system with bilinear outputs $\system_{\Rvec} =(\A, \B, \Rvec)$, the $\Ltwo$-$\Ltwo$ gain can be bounded by 
  \begin{equation*}
    \funNorm{y}_{\timeSpacetwo{p}}^2 \le \left( \frac{1}{2\pi}\sum_{i=1}^p \funNorm{\Rm_i (s\I-\A)^{-1}\B}_{\freqSpacetwoC{m\times m}}^2 \right) \funNorm{u}_{\timeSpacetwo{m}}^4, \label{eq:RP}
  \end{equation*}
  where $\Rm_i = \vtf( \Rvec^\T \mathbf{e}_i) \in \R^{\inputDim \times \stateDim}$.
\end{theorem}

\begin{proof}
  We analyze the output first in the time domain and use that we can compute each component of the output $\mathbf{e}_i^\T\outputVar(\tau_1, \tau_2)$ as 
  \begin{align}
    \mathbf{e}_i^\T\outputVar(t) &=  \int_0^t \int_0^t \mathbf{e}_i^\T \Rvec \left[\mexp{\A \tau_1} \B \otimes \I \delta(\tau_2)\right]\left[\inpVar(t-\tau_1) \otimes \inpVar(t-\tau_2)\right] \dtau_2 \dtau_1 \notag\\
    &=  \int_0^t \vecc(\Rm_i)^\T \left[\mexp{\A \tau_1} \B \otimes \I \right]\left[\inpVar(t-\tau_1) \otimes \inpVar(t)\right] \dtau_1 \notag\\
    &= \int_0^t \vecc(\Rm_i)^\T\left[\mexp{\A \tau_1} \B\inpVar(t-\tau_1)  \otimes \inpVar(t) \right]\dtau_1\notag\\
    &= \int_0^t \vecc\left(\inpVar^\T(t) \Rm_i \mexp{\A \tau_1} \B \inpVar(t-\tau_1)\right) \dtau_1 \notag \\
    &= \int_0^t \inpVar^\T(t) \Rm_i \mexp{\A \tau_1} \B \inpVar(t-\tau_1) \dtau_1 \label{eq:tobeadapted}
  \end{align}
  where we used for the last equality that the argument of the vectorization is a scalar. For the norm computation we then yield
  \begin{align*}
    \funNorm{\outVar}^2_{\timeSpacetwo{p}} &= \int_0^\infty \vecNorm{\sum_{i=1}^p \mathbf{e}_i \inpVar^\T(t) \left(\int_0^t  \Rm_i \mexp{\A \tau_1} \B \inpVar(t-\tau_1) \dtau_1 \right)}_2^2 \dt  \\ 
    &\!\!\!\overset{(CS)}{\le} \int_0^\infty \sum_{i=1}^p \vecNorm{\Rm_i\int_0^t \mexp{\A \tau_1}\B \inpVar(t-\tau_1)\dtau_1}_2^2 \vecNorm{\inpVar(t)}_2^2 \dt \\
    &\le \left( \sup_{t \in [0, \infty)}  \sum_{i=1}^p \vecNorm{\Rm_i\int_0^t \mexp{\A \tau_1}\B \inpVar(t-\tau_1)\dtau_1 }_2^2 \right) \funNorm{\inpVar}_{\timeSpacetwo{m}}^2,
  \end{align*} 
  where we employed the Cauchy-Schwarz inequality in $(CS)$ and the monotonicity of the integral in the last inequality. Further computation yields by application of the triangle $\trineg$ and the Cauchy-Schwarz inequality $(CS)$
  \begin{align*}
    \vecNorm{\int_0^t \Rm_i \mexp{\A \tau} \B \inpVar(t-\tau) \dtau}_2 &\overset{\trineg}{\le}\int_0^t \vecNorm{\Rm_i \mexp{\A \tau} \B \inpVar(t-\tau)}_2 \dtau \\
    & \overset{}{\le}\int_0^t \vecNorm{\Rm_i \mexp{\A \tau}\B}_2 \vecNorm{\inpVar(t-\tau)}_2 \dtau \\
    & \!\!\overset{(CS)}{\le}\left(\int_0^t \vecNorm{\Rm_i \mexp{\A \tau}\B}^2_2 \dtau \right)^{\tfrac{1}{2}}\left( \int_0^t \vecNorm{\inpVar(t-\tau)}^2_2 \dtau \right)^{\tfrac{1}{2}}.
  \end{align*}
  From here, it is clear that we can estimate the supremum over $t \in [0, \infty)$ by 
  \begin{align*}
    &\phantom{=}\sup_{t \in [0, \infty)}\sum_{i=1}^p \vecNorm{\Rm_i\int_0^t \mexp{\A \tau}\B \inpVar(t-\tau)\dtau }_2^2 \notag \\ &\le \sum_{i=1}^p \left(\int_0^\infty \vecNorm{\Rm_i \mexp{\A \tau}\B}^2_2 \dtau \right)\left( \int_0^\infty \vecNorm{\inpVar(\tau)}^2_2 \dtau \right) \notag \\
    &\overset{(P)}{=} \left( \frac{1}{2\pi}\sum_{i=1}^p \funNorm{\Rm_i (s\I-\A)^{-1}\B}_{\freqSpacetwoC{m\times m}}^2 \right) \funNorm{u}_{\timeSpacetwo{m}}^2,
  \end{align*}
  which yields the desired result.
\end{proof}

\begin{remark}
  In the result \Cref{thm:onlyR} with one output $p=1$ one can identify the $\Ltwo$-$\Ltwo$-gain bound as the $\Htwo$-norm of an \LTI system $\system_\Rm = (\A, \B, \Rm)$.
\end{remark}

\section{Quadratic throughput}
\label{sec:throughput}

Finally, we consider an \LTI system with outputs depending only quadratically on $\inpVar$
\begin{equation}\label{eq:onlyP}
  \system_{\Pvec} \; \left\lbrace \;\begin{aligned}
    \dot{\state}(t) &= \A \state(t) + \B \inpVar(t), \\
    \outputVar(t) &=  \Pvec (\inpVar(t) \otimes \inpVar(t)).
  \end{aligned} \right.
\end{equation} 
We observe in the following that assuming $u\in \timeSpacetwo{m}$ is not sufficient to have a bounded output in the $\Ltwo$-norm. To clarify this, we compute the $\Ltwo$-norm of the output as
\begin{align*}
  \funNorm{\outVar}_{\timeSpacetwo{p}}^2 &= \int_0^\infty \vecNorm{\Pvec (\inpVar(t) \otimes \inpVar(t))}_2^2 \dt \\
  &\le \vecNorm{\Pvec}_2^2 \int_0^\infty \vecNorm{\inpVar(t) \otimes \inpVar(t)}_2^2 \dt, 
\end{align*}  
Since $\vecNorm{\inpVar(t) \otimes \inpVar(t)}_2 = \vecNorm{\inpVar(t)}_2^2$, the remaining integral is the fourth power of the $\Lfour$-norm of the input, which is in general not bounded by its $\Ltwo$-norm. This can be illustrated by the input $\inpVar_\varepsilon(t) = \varepsilon^{-\tfrac{1}{2}} \chi_{[0, \varepsilon]}(t)$, for which
\begin{equation*}
  \funNorm{\inpVar_\varepsilon}_{\timeSpacetwo{}}^2 = 1, \quad \text{while} \quad \funNorm{\inpVar_\varepsilon}_{\timeSpacefour{}}^4 = \varepsilon^{-1}.
\end{equation*}
The latter is unbounded for $\varepsilon \rightarrow 0$ and hence, with $m=p=1$ and $\Pm_1 = 1$,
\begin{equation*}
  \funNorm{\outVar_\varepsilon}_{\timeSpacetwo{}}^2 = \funNorm{\inpVar_\varepsilon}_{\timeSpacefour{}}^4 = \frac{1}{\varepsilon} \rightarrow \infty
\end{equation*}
while $\funNorm{\inpVar_\varepsilon}_{\timeSpacetwo{}}^2 = 1$. Therefore, no $\Ltwo$-$\Ltwo$-gain bound can hold when the output includes a quadratic throughput. Instead, there exists a constant $c_{2, \Pvec} \in \R_+$ such that
\begin{equation*}
  \funNorm{\outVar}_{\timeSpacetwo{p}} \le c_{2,\Pvec}\funNorm{\inpVar}_{\timeSpacefour{m}}^2 \le c_{2, \Pvec}\funNorm{\inpVar}_{\timeSpacetwo{m}}\funNorm{\inpVar}_{\timeSpaceinf{m}}.
\end{equation*}

Hence, presuming that  $u \in \timeSpacetwo{m}\cap \timeSpaceinf{m}$ is a natural assumption for computing gain bounds with quadratic throughput terms. We compute this also involving bilinear terms, i.e., considering systems of type 
\begin{equation}\label{eq:PR}
  \system_{\Rvec, \Pvec} \; \left\lbrace \;\begin{aligned}
    \dot{\state}(t) &= \A \state(t) + \B \inpVar(t), \\
    \outputVar(t) &=  \Rvec (\state(t) \otimes \inpVar(t))+ \Pvec (\inpVar(t) \otimes \inpVar(t)).
  \end{aligned} \right.
\end{equation} 

Applying the Hölder inequality to a suitable adaptation of \eqref{eq:tobeadapted} we yield for the norm of the output of  $\system_{\Rvec, \Pvec}$ 
\begin{align*}
  \funNorm{\outVar}^2_{\timeSpacetwo{p}} &\le \left(\int_0^\infty \sum_{i=1}^p \vecNorm{\int_0^t \left(\Rm_i\mexp{\A \tau}\B  + \Pm_i \delta(\tau)\right)\inpVar(t-\tau)\dtau}_2^2 \dt \right) \funNorm{\inpVar}_{\timeSpaceinf{m}}^2.
\end{align*} 
Now, one can apply standard $\Hinf$-results from \LTI theory for each of the outputs
\begin{align*}
  & \phantom{=}\vecNorm{\int_0^t \left(\Rm_i\mexp{\A \tau}\B  + \Pm_i \delta(\tau)\right)\inpVar(t-\tau)\dtau}_{\timeSpacetwo{m}}\le \funNorm{\Sigma_{\LTI}(\A, \B, \Rm_i, \Pm_i)}_{\Hinf} \funNorm{\inpVar}_{\timeSpacetwo{m}},
\end{align*}
where $\funNorm{\Sigma_{\LTI}(\A, \B, \Rm_i, \Pm_i)}_{\Hinf}$ denotes the $\Hinf$-norm of the \LTI system determined by $(\A, \B, \Rm_i, \Pm_i)$. Summarized this then yields 
\begin{equation*}
  \funNorm{\outVar}_{\timeSpacetwo{p}} \le \sqrt{\sum_{i=1}^p \funNorm{\Sigma_{\LTI}(\A, \B, \Rm_i, \Pm_i)}^2_{\Hinf}} \funNorm{u}_{\timeSpacetwo{m}} \funNorm{u}_{\timeSpaceinf{m}}.
\end{equation*}

\section{Linear, quadratic, bilinear and throughput outputs}
\label{sec:bilinear_and_throughput}

Finally, we can summarize the $\Ltwo$-$\Ltwo$-gain bound for \LQO systems \eqref{eq:base_ext} without quadratic throughput, i.e., $\Pvec = 0$, in the following theorem.
\begin{theorem}\label{thm:full}
  The $\Ltwo$-$\Ltwo$ gain of an \LQO system \eqref{eq:base_ext} with $\Pvec = 0$, can be bounded by 
  \begin{multline*}
    \funNorm{\outVar}_{\timeSpacetwo{\outVarDim}} \le \sup_{s \in \imagUnit \R} \sigma_{\max} \left[ \Cm (s\I-\A)^{-1} \B + \D \right] \funNorm{ \inpVar}_{\timeSpacetwo{\inpVarDim}} \\ 
    + \frac{1}{\sqrt{2\pi}} \left( \sqrt{ \sum_{i=1}^p \funNorm{\Rm_i (s\I-\A)^{-1}\B}^2_{\freqSpacetwoC{m\times m}}}  \right.\\
    \left. +  \sqrt{\sum_{i=1}^p \sup_{s\in \imagUnit \R} \funNorm{\widetilde \transferFunction_{\Mvec, i}(\cdot, s-\cdot)}_{\freqSpacetwoC{\inpVarDim^2}}^2 }\right) \funNorm{\inpVar}_{\timeSpacetwo{\inpVarDim}}^2.
  \end{multline*}
  Here, we denote $\Rm_i = \vtf(\mathrm{e}_i^\T \Rvec)$,  $\sigma_{\max}$ refers to the largest singular value, and $\widetilde{\transferFunction}_{\Mvec, i} = \mathbf{e}_i^\T \widetilde{\transferFunction}_{\Mvec}$ with the latter given by \eqref{eq:tf_only_M}.
\end{theorem}

\begin{proof}
  We split the summand using the triangle inequality
  \begin{align*}
    \funNorm{\outVar}_{\timeSpacetwo{\outVarDim}} & = \funNorm{\outVar _{\mathrm{lin}} + \outVar _{\mathrm{\Mvec}}  +\outVar _{\mathrm{\Rvec}}}_{\timeSpacetwo{\outVarDim}} \\ &\le \funNorm{\outVar _{\mathrm{lin}}}_{\timeSpacetwo{\outVarDim}} + \funNorm{\outVar _{\Mvec} }_{\timeSpacetwo{\outVarDim}} + \funNorm{\outVar _{\Rvec} }_{\timeSpacetwo{\outVarDim}},
  \end{align*}
  where the summands are given by 
  \begin{align*}
    \outVar_{\mathrm{lin}}(t) &= \int_0^t \left[ \Cm \mexp{\A \tau} \B + \D \delta(\tau) \right] \inpVar(t-\tau) \dtau, \\
    \outVar_{\Mvec}(t) &= \int_0^t  \int_0^t \left[\Mvec \left( \mexp{\A \tau_1} \B \otimes \mexp{\A \tau_2} \B \right) \right] \cdot [\inpVar(t-\tau_1) \otimes \inpVar(t-\tau_2)] \dtau_1 \dtau_2.\\
    \outVar_{\Rvec}(t) &= \int_0^t  \int_0^t \left[\Rvec\left( \mexp{\A \tau_1} \B \otimes \I\delta(\tau_2)\right) \right] \cdot [\inpVar(t-\tau_1) \otimes \inpVar(t-\tau_2)] \dtau_1 \dtau_2.
  \end{align*}
  For the linear part the relationship 
  \begin{equation}
    \funNorm{ \outVar_{\mathrm{lin}} }_{\timeSpacetwo{\outVarDim}} \le \sup_{\sigma \in \imagUnit \R} \sigma_{\max} \left[ \Cm (\sigma\I-\A)^{-1} \B + \D \right]  \funNorm{\inpVar}_{\timeSpacetwo{\inpVarDim}} \label{eq:CD}
  \end{equation}
  is well established in literature, see \cite{Ant05}. The quadratic component is bounded by \Cref{thm:central} and the bilinear component by \Cref{thm:onlyR}. Summing the contributions yields the desired result.
\end{proof}

\section{Example}
\label{sec:example}

We illustrate the new bound with a small numerical example of state dimension $\stateDim = 3$
\begin{equation}\label{eq:example}
  \A = \begin{bmatrix}
    -1 & 2 &0\\-1 & -2 & 1 \\ 0 &-3 & -4
  \end{bmatrix}, \quad \B =\begin{bmatrix}
    1 \\ 0 \\ 1
  \end{bmatrix}, \quad \Mvec = \vecc\left(\begin{bmatrix}
    1 & 0 & 0 \\ 0 & 1 & 0 \\ 0 & 0 & 1
  \end{bmatrix}\right)^\T.
\end{equation}
From here, we can compute the gain bound using the Sylvester-equation formulation \Cref{thm:lyap_formulation} and taking the supremum over all $\omega$
\begin{equation*}
  \funNorm{\outVar}_{\timeSpacetwo{}} \le 0.43 \funNorm{\inpVar}_{\timeSpacetwo{}}^2.
\end{equation*}
Using the control systems toolbox \cite{FulGMMP21} in \abbr{Python} we also computed the $\Ltwo$-$\Ltwo$ gain $\gamma \coloneqq \tfrac{\funNorm{\outVar}_{\timeSpacetwo{}}}{\funNorm{\inpVar}_{\timeSpacetwo{}}^2}$ for various inputs as listed in \Cref{tab:gains} and can thereby validate the gain bound for this simple example and the given inputs.
\begin{table}[]
  \begin{tabular}{ll||ll}
  Input $\inpVar(t)$ & Gain $\gamma$ & Input $\inpVar(t)$ & Gain $\gamma$  \\ \hline
  $\mexp{-t}$ & 0.284 & $\mexp{-t}- \mexp{-4t}$ &  0.324\\
  $\mexp{-3t}$ & 0.210 & $\mexp{-(t-3)^2}$&  0.361\\
  $\mexp{-0.3t}\sin(2t)$&0.214  & $\sin^2(\pi t)\chi_{[0,1]}(t)$  & 0.264\\
  $\mexp{-0.5t}\cos(5t)$& 0.047 & $\mexp{-0.4 t}  \sin(2.5 t^2)$ & 0.130
  \end{tabular}
  \caption{$\Ltwo$-$\Ltwo$-gains of \eqref{eq:example} computed for various inputs.}
  \label{tab:gains}
\end{table}

\section{Conclusion}
\label{sec:discussion}

We have established three main results.  First, we showed that the $\Ltwo$-$\Ltwo$ gain of a purely quadratic-output system is bounded in terms of the supremum over $\omega\in\R$ of the $\Ltwo$-norms of the bivariate transfer function restricted to the anti-diagonals 
\(\mathcal D=\{(s,\,\imagUnit \omega-s)\mid s\in i\mathbb R\}\).  
The second result showed that these norms can be computed by solving a small set of Sylvester equations, thus avoiding any numerical integration over the frequency axis. The validity of the bound was illustrated by a small numerical example. 
Third, we showed that the full $\Ltwo$-$\Ltwo$-gain bound can be written in terms of the previous results, the $\Hinf$-norm of the transfer function of the linear part, and a term resembling the $\Htwo$-norm of an \LTI system for the bilinear part. The quadratic throughput term can therein only be included if one considers an $\Ltwo\cap\Linfty$-$\Ltwo$-gain bound. Future research directions include the development of a model order reduction algorithm which aims to minimize the $\Ltwo$-$\Ltwo$-error.
 
\vspace*{-0.1cm}
\removeAuthorData{
\subsection*{Acknowledgements} The author thanks Prof.~Benjamin Unger and Mattia Manucci for a careful reading of the draft and for the valuable comments that improved the presentation of the results. The author acknowledges funding by the Deutsche Forschungsgemeinschaft (DFG, German Research Foundation) – Project-ID 258734477 – SFB 1173 and by the International Max Planck Research School for Intelligent Systems (IMPRS-IS). 
}

\vspace*{-0.12cm}
\subsection*{Declaration of generative AI and AI-assisted technologies in the manuscript preparation process} During the preparation of this work the author used GPT-OSS 120B hosted at Karlsruhe Institute of Technology in order to improve grammar and readability. After using this tool, the author reviewed and edited the content as needed and takes full responsibility for the content of the published article.


\bibliographystyle{_plain-doi} 
\bibliography{extracted.bib}

@article{HolNSU25,
  author =        {Holicki, Tobias and Nicodemus, Jonas and
                   Schwerdtner, Paul and Unger, Benjamin},
  journal =       {{SIAM} J. Cont. Optim.},
  month =         jun,
  number =        {3},
  pages =         {2154--2176},
  publisher =     {Society for Industrial and Applied Mathematics},
  title =         {Energy Matching in Reduced Passive and
                   Port-{Hamiltonian} Systems},
  volume =        {63},
  year =          {2025},
  doi =           {10.1137/23m1600931},
  issn =          {1095-7138},
}

@article{EngW08,
  author =        {Engwerda, Jacob and Weeren, Arie},
  journal =       {Automatica},
  number =        {1},
  pages =         {265-271},
  title =         {A result on output feedback linear quadratic control},
  volume =        {44},
  year =          {2008},
  doi =           {10.1016/j.automatica.2007.04.025},
  issn =          {0005-1098},
}

@article{HaaUW13,
  author =        {Haasdonk, Bernard and Urban, Karsten and
                   Wieland, Bernhard},
  journal =       {{SIAM}-{ASA} J. Uncertain. Quantif.},
  number =        {1},
  pages =         {79--105},
  publisher =     {Society for Industrial and Applied Mathematics},
  title =         {Reduced basis methods for parameterized partial
                   differential equations with stochastic influences
                   using the {Karhunen--Lo{e`}ve} expansion},
  volume =        {1},
  year =          {2013},
  doi =           {10.1137/120876745},
}

@article{Sch92,
  author =        {van der Schaft, Arjan},
  journal =       {{IEEE} Trans. Automat. Control},
  number =        {6},
  pages =         {770--784},
  publisher =     {{IEEE}},
  title =         {$L_2$-gain analysis of nonlinear systems and
                   nonlinear state feedback {$H_\infty$} control},
  volume =        {37},
  year =          {1992},
  doi =           {10.1109/9.256331},
}

@book{ZhoDG96,
  author =        {Zhou, Kemin and Doyle, John C. and Glover, Keith},
  publisher =     {Prentice-Hall},
  title =         {Robust and optimal control},
  year =          {1996},
  isbn =          {0134565673},
  url =           {https://dl.acm.org/doi/book/10.5555/225507},
}

@book{BenSGQR21,
  address =       {Berlin/Boston},
  author =        {Peter Benner and Wil Schilders and
                   Stefano Grivet-Talocia and Alfio Quarteroni and
                   Gianluigi Rozza and Luís Miguel Silveira},
  number =        {1},
  publisher =     {De Gruyter},
  series =        {Model Order Reduction},
  title =         {Volume 1: System- and Data-Driven Methods and
                   Algorithms},
  year =          {2021},
  isbn =          {9783110497717},
}

@article{BenGD21,
  author =        {Benner, Peter and Goyal, Pawan and Duff, Igor Pontes},
  journal =       {{IEEE} Trans. Automat. Control},
  number =        {2},
  pages =         {886--893},
  publisher =     {{IEEE}},
  title =         {Gramians, Energy Functionals, and Balanced Truncation
                   for Linear Dynamical Systems With Quadratic Outputs},
  volume =        {67},
  year =          {2021},
  doi =           {10.1109/TAC.2021.3086319},
}

@article{ReiGPG26,
  author =        {Reiter, Sean and Gosea, Ion Victor and
                   Pontes Duff, Igor and Gugercin, Serkan},
  journal =       {{SIAM} J. Matrix Anal. Appl.},
  number =        {3},
  pages =         {1679--1714},
  publisher =     {Society for Industrial and Applied Mathematics},
  title =         {H2-optimal model reduction of linear quadratic-output
                   systems by multivariate rational interpolation},
  volume =        {47},
  year =          {2026},
  doi =           {10.1137/25m1766735},
  issn =          {1095-7162},
}

@article{DiaHGA23,
  author =        {Diaz, Alejandro N. and Heinkenschloss, Matthias and
                   Gosea, Ion Victor and Antoulas, Athanasios C.},
  journal =       {Adv. Comput. Math.},
  number =        {6},
  pages =         {95},
  publisher =     {Springer-Verlag},
  title =         {Interpolatory model reduction of quadratic-bilinear
                   dynamical systems with quadratic-bilinear outputs},
  volume =        {49},
  year =          {2023},
  doi =           {10.1007/s10444-023-10096-2},
  isbn =          {1572-9044},
}

@article{Deb15,
  author =        {Debnath, Lokenath},
  journal =       {Internat. J. Appl. Comput. Math.},
  month =         apr,
  number =        {2},
  pages =         {223--241},
  publisher =     {Springer-Verlag},
  title =         {The Double {Laplace} Transforms and Their Properties
                   with Applications to Functional, Integral and Partial
                   Differential Equations},
  volume =        {2},
  year =          {2015},
  doi =           {10.1007/s40819-015-0057-3},
  issn =          {2199-5796},
}

@article{CheC73,
  author =        {Chen, Chihfan and Chiu, R. F.},
  journal =       {Internat. J. Syst. Sci.},
  month =         jul,
  number =        {4},
  pages =         {647--660},
  publisher =     {Informa UK Limited},
  title =         {New theorems of association of variables in multiple
                   dimensional {Laplace} transform},
  volume =        {4},
  year =          {1973},
  doi =           {10.1080/00207727308920045},
  issn =          {1464-5319},
}

@article{ZhaW15,
  author =        {Zhang, Yang and Wong, Ngai},
  journal =       {Int J. Circ. Theor. Appl.},
  month =         nov,
  number =        {7},
  pages =         {1367--1384},
  publisher =     {Wiley},
  title =         {Compact model order reduction of weakly nonlinear
                   systems by associated transform},
  volume =        {44},
  year =          {2015},
  doi =           {10.1002/cta.2165},
  issn =          {1097-007X},
}

@misc{PetP12,
  author =        {Petersen, Kaare Brandt and Pedersen, Michael Syskind},
  month =         {nov},
  note =          {Version 20121115},
  publisher =     {Technical University of Denmark},
  title =         {The Matrix Cookbook},
  year =          {2012},
  url =           {http://www2.imm.dtu.dk/pubdb/p.php?3274},
}

@book{Ant05,
  author =        {Antoulas, Athanasios C.},
  publisher =     {Society for Industrial and Applied Mathematics},
  title =         {Approximation of large-scale dynamical systems},
  year =          {2005},
  doi =           {10.1137/1.9780898718713},
}

@inproceedings{FulGMMP21,
  author =        {Fuller, Sawyer and Greiner, Ben and Moore, Jason and
                   Murray, Richard and van Paassen, Rene and
                   Yorke, Rory},
  booktitle =     {2021 60th IEEE Conference on Decision and Control
                   (CDC)},
  month =         dec,
  pages =         {4875--4881},
  publisher =     {IEEE},
  title =         {The {Python} Control Systems Library
                   (python-control)},
  year =          {2021},
  doi =           {10.1109/cdc45484.2021.9683368},
}

@book{HorJ91,
  author =        {Horn, Roger A. and Johnson, Charles R.},
  month =         apr,
  publisher =     {Cambridge University Press},
  title =         {Topics in Matrix Analysis},
  year =          {1991},
  doi =           {10.1017/cbo9780511840371},
  isbn =          {9780511840371},
}


\appendix
\section{Proof of \Cref{lem:partial-fraction-decomposition}}
\label{app:partial-fraction}

\begin{proof}
  Let the eigenvalues of $\A$ be given by $\lambda_1, ..., \lambda_\stateDim$, then every eigenvalue $\lambda^\oplus_k$ for $k \in \{1,..., n^2\}$ of $\I \otimes (\A^\T - \imagUnit \omega \I) + \A^\T \otimes \I$ can be written as $\lambda^\oplus_k = \lambda_i - \imagUnit \omega + \lambda_j $ for $i,j \in \{1,...,n\}$ \cite[4.4.5]{HorJ91}. Hence, the matrix $\I \otimes (\A^\T -\imagUnit \omega \I) + \A^\T \otimes \I$ is Hurwitz and thus invertible. 
  Computing then yields 
  \begin{align*}
    &\phantom{=}  \left[(s\I - \A^\T)^{-1} \otimes ((\imagUnit \omega -s) \I - \A^\T)^{-1}\right] \left[\I \otimes (\A^\T- \imagUnit \omega \I ) + \A^\T \otimes \I\right] \\ 
    &= \left[(s \I - \A^\T)^{-1} \otimes  ((\imagUnit \omega-s)\I -\A^\T)^{-1}(\A^\T - \imagUnit \omega\I)\right] + \left[ (s \I - \A^\T)^{-1} \A^\T\otimes  ((\imagUnit \omega-s)\I -\A^\T)^{-1}\right]\\ 
    &= \left[(s \I - \A^\T)^{-1} \otimes  ((\imagUnit \omega-s)\I -\A^\T)^{-1}\left( - ((\imagUnit \omega - s)\I - \A^\T )  -s\I \right) \right]\\ 
    &\phantom{=} + \left[ (s \I - \A^\T)^{-1} \left( -(s\I - \A^\T) + s\I\right) \otimes  ((\imagUnit \omega-s)\I -\A^\T)^{-1}\right]\\ 
    &= -(s \I - \A^\T)^{-1} \otimes \I - s(s \I - \A^\T)^{-1} \otimes  ((\imagUnit \omega-s)\I -\A^\T)^{-1} \\ 
    &\phantom{=}- \I \otimes  ((\imagUnit \omega-s)\I -\A^\T)^{-1} + s (s \I - \A^\T)^{-1} \otimes  ((\imagUnit \omega-s)\I -\A^\T)^{-1}\\ 
    &= -(s \I - \A^\T)^{-1} \otimes \I - \I \otimes  ((\imagUnit \omega-s)\I -\A^\T)^{-1} 
   \qedhere
  \end{align*}
\end{proof}

\end{document}